\documentclass{amsart}    
\usepackage{graphicx}
\usepackage{hyperref}
\usepackage{amssymb}
\usepackage{setspace}

\title{Higher-order Analogues of the Slice Genus of a Knot}

\author{Peter D. Horn}
\address{Department of Mathematics\\Columbia University\\MC 4403\\2990 Broadway\\New York, NY 10027}
\email{pdhorn@math.columbia.edu}
\urladdr{http://math.columbia.edu/~pdhorn}

\newtheorem{thm}{Theorem}[section]    
\newtheorem{lem}[thm]{Lemma}          
\newtheorem{cor}[thm]{Corollary}
\newtheorem{prop}[thm]{Proposition}
\newtheorem*{mainresultthm}{Theorem~\ref{arbhinog}}
\newtheorem*{mainresult}{Corollary~\ref{nogisbetterthanslicegenus}}      
\newtheorem*{secondaryresult}{Theorem~\ref{sigboundgen}}
\theoremstyle{definition}
\newtheorem{defn}[thm]{Definition}    
\newtheorem{exam}[thm]{Example}
\newtheorem*{rem}{Remark}             
\newcommand{\inv}{^{-1}}
\newcommand{\p}{\pi_1}
\newcommand{\del}{\partial}
\renewcommand{\a}{\alpha}
\renewcommand{\b}{\beta}

\newcommand{\Gn}[1]{G^{(#1)}}
\newcommand{\pr}[2]{\p({#1})_r^{(#2)}}
\newcommand{\Gr}[2]{{#1}_r^{(#2)}}
\renewcommand{\H}[2]{H_{#1}\left({#2}\right)}
\newcommand{\pn}[2]{\p\left({#1}\right)^{({#2})}}
\newcommand{\F}{{\mathcal{F}}}
\newcommand{\C}{{\mathcal{C}}}
\newcommand{\R}{{\mathbb{R}}}
\newcommand{\Z}{{\mathbb{Z}}}
\newcommand{\Q}{{\mathbb{Q}}}
\newcommand{\G}{{\mathcal{G}}}
\newcommand{\N}{{\mathbb{N}}}

\begin{document}


\begin{abstract}    
	For certain classes of knots we define geometric invariants called higher-order genera.  Each of these invariants is a refinement of the slice genus of a knot.  We find lower bounds for the higher-order genera in terms of certain von Neumann $\rho$-invariants, which we call higher-order signatures.  The higher-order genera offer a refinement of the Grope filtration of the knot concordance group.
\end{abstract}

\maketitle

\section{Introduction}

\footnotetext{The published version of this paper may be accessed at \url{http://imrn.oxfordjournals.org/cgi/content/abstract/rnq115?ijkey=ZIkXMttJWMD1Uz8&keytype=ref.}}

A \emph{knot} is an embedding of the circle into the three-sphere.  All embeddings are required to be topologically flat or smooth.  Two knots $K_{0}, K_{1}$ are \emph{concordant} if there is an annulus $A$ embedded in $S^{3}\times[0,1]$ in such a way that $A\cap\left(S^{3}\times\{i\}\right)=K_{i}$ for $i=0,1$.  If a knot $K$ is concordant to the unknot, we call $K$ a \emph{slice knot}.  Given two knots, one can ``add'' them via the connected sum operation $\#$, defined in~\cite{dR76}.  Equipped with the connected sum operation, the set of knots modulo (topological or smooth) concordance forms the (topological or smooth) \emph{knot concordance group} $\C$.  The class of slice knots serves as the identity element of this group.

Cochran, Orr and Teichner have introduced two filtrations of the topological knot concordance group $\C$~\cite{COT03}.  The $(n)$-solvable filtration $$\cdots\subset \F_{n.5}\subset \F_{n}\subset\cdots\subset\F_{1.5}\subset\F_{1}\subset\F_{0.5}\subset\F_{0}\subset \C$$  is defined in terms of algebraic properties on the second homology of certain $4$-manifolds, each of whose boundary is $0$-surgery on a knot.  The Grope filtration $$\cdots\subset \G_{n+2.5}\subset \G_{n+2}\subset\cdots\subset\G_{3}\subset\G_{2.5}\subset\G_{2}\subset \C$$ is defined much more geometrically.  Rigorous definitions of these filtrations will be provided below.  These filtrations are related to one another in the sense that $\G_{n+2}\subset\F_{n}$ for all $n\in\frac{1}{2}\N$~\cite[Theorem 8.11]{COT03}.  Recently, Cochran, Harvey and Leidy proved that $\F_{n}/\F_{n.5}$ has infinite rank for all $n$~\cite{CHL07b}.  Subsequently, the author proved the analogous result for the Grope filtration~\cite{pH08}.  These results were proven using signatures of certain $4$-manifolds.  While algebraic techniques are appropriate when working with the $(n)$-solvable filtration, they do not reflect the geometric nature of the Grope filtration.  The main focus of this paper is to define a geometric invariant that will distinguish knots in $\G_{n+2}$.

Given a knot $K$, the \emph{slice genus} of $K$ is the minimal genus of surfaces embedded in $D^{4}$ with boundary equal to $K\subset S^{3}=\del D^{4}$.  The slice genus is a concordance invariant of $K$.  In the spirit of the Cochran-Orr-Teichner filtrations of $\C$, we introduce a series of refinements of the slice genus.  For knots in $\G_{n+2}$, we will define a concordance invariant called the $n^{\textrm{th}}$-order genus.  Our main result is that the $n^{\textrm{th}}$-order genus distinguishes knots in $\G_{n+2}$ that are not distinguished by the slice genus.  That is, each of our higher-order genera is a refinement of the notion of slice genus.

\begin{mainresultthm}
	For any $n\geq 1$, there is a fixed $g$ and a knot in $\G_{n+2}$ with slice genus bounded above by $g$ and arbitrarily high $n^{\textrm{th}}$-order genus.  Furthermore, this knot has infinite order in $\G_{n+2}/\F_{n.5}$.
\end{mainresultthm}

\begin{mainresult}
	For any $n\geq 1$, there are infinitely many knots that lie in $\G_{n+2}$ whose slice genera are equal but whose $n^{\textrm{th}}$-order genera are distinct.
\end{mainresult}

Murasugi proved~\cite[Theorem 9.1]{kM65} that the ordinary signature of a knot is a lower bound for the slice genus of that knot (henceforth ``Murasugi's inequality'').  Gilmer later proved~\cite[Theorem 1]{pG82} that the sum of certain Casson-Gordon invariants and the ordinary signature bounds the slice genus from below (henceforth ``Gilmer's inequality'').  Cochran, Orr and Teichner first used $L^{2}$-signatures to study knots.  First, we define higher-order analogues of slice genus, and to any $(n)$-solvable knot we assign a set of real numbers, called the $n^{\textrm{th}}$-order signatures.  This begs the question of whether there is a higher-order analogue of Murasugi's inequality.  Our primary tool is the desired higher-order analogue.

\begin{secondaryresult}
	If $K\in\G_{n+2}$, there is an $n^{\textrm{th}}$-order signature of $K$ that gives a lower bound for the $n^{\textrm{th}}$-order genus of $K$.
\end{secondaryresult}

We are not the first to utilize $L^{2}$-signatures in the study of genus-like invariants.  Cha used metabelian $L^{2}$-signatures to obtain new lower bounds on the minimal genus of embedded surfaces representing a given $2$-dimensional homology class in certain $4$-manifolds~\cite{jC08}.  An application of Cha's methods was to find bounds for the slice genus of knots~\cite[Proposition 5.1]{jC08}.  Our Theorem~\ref{sigboundgen} uses the $L^{2}$-signatures to obtain lower bounds for the higher-order genera.  While Cha obtained obstructions to slice genus, we obtain higher-order obstructions to the higher-order genera.  It seems that the only (classical) sliceness obstruction our higher-order genera give is that if one of the higher-order genera of a knot is positive, then that knot cannot be slice.  However, a knot having large higher-order genera does not in general obstruct the knot from having a small (but positive) slice genus.

We should note that our higher-order signatures give a lower bound on the topological higher-order genera and often fail to be accurate in the smooth category.  Consequently, we choose to work in the topological category, except for Section~\ref{smoothcategory}, which contains examples in the smooth category.

\section{Definitions}

We start with the geometric definitions.

\begin{defn}~\cite{FT95}
	A \textbf{grope} is a special pair (2-complex, base circle).  A grope has a \textbf{height} $n\in\frac{1}{2}\mathbb{N}$.  A grope of height $1$ is precisely a compact, oriented surface $\Sigma$ with a single boundary component (the base circle).  For $n\in\N$, a grope of height $n+1$ is defined recursively as follows: let $\{\a_i,\b_{i}: i=1,\ldots,g\}$ be a symplectic basis of curves for $\Sigma$, the first stage of the grope.  Then a grope of height $n+1$ is formed by attaching gropes of height $n$ to each $\a_i$ and $\b_{i}$ along the base circles.
	
	A grope of height $1.5$ is formed by attaching gropes of height $1$ (i.e. surfaces) to a Lagrangian of a symplectic basis of curves for $\Sigma$.  That is, a grope of height $1.5$ is a surface with surfaces glued to ``half'' of the basis curves.  In general, a grope of height $n+1.5$ is obtained by attaching gropes of height $n$ to the $\a_{i}$ and gropes of height $n+1$ to the $\beta_{i}$.

	Given a 4-manifold $W$ with boundary $M$ and a framed circle $\gamma\subset M$, we say that $\gamma$ bounds a \textbf{Grope} in $W$ if $\gamma$ extends to an embedding of a grope with its untwisted framing.  That is, a Grope has a trivial normal bundle, so parallel push-offs can be taken.  Knots in $S^3$ are always equipped with the zero framing.
\end{defn}

The set of all concordance classes of knots that bound Gropes of height $n$ in $D^4$ is denoted $\G_n$, which is a subgroup of $\C$.  We may choose to forget the top stages of a Grope.  Thus, if $K$ bounds a Grope of height $n+1$ in $D^{4}$, $K$ also bounds a Grope of height $n$ in $D^{4}$.  We see that $\G_{n+1}\subset\G_{n}$ as subgroups of $\C$, and this series of subgroups is \textbf{the Grope filtration of the knot concordance group}.  By `$K\in\G_{n}$,' we mean a knot $K$ whose concordance class lies in $\G_{n}$, or equivalently, a knot that bounds a Grope of height $n$ in $D^{4}$.

\begin{defn}
	For $K\in\G_{n+2}$, define the \textbf{$n^{\textrm{th}}$-order genus of $K$} to be the minimum of the genera of the first stage surfaces of Gropes of height $n+2$ in $D^4$ bounded by $K$.  Denote the $n^{\textrm{th}}$-order genus of $K$ by $g_n(K)$.  With this numbering scheme, the slice genus of $K$ is the $-1^{\textrm{st}}$-order genus of $K$.
\end{defn}

It is immediately clear that for $K\in\G_{n+2}$, $0\leq g_{-1}(K)\leq g_0(K)\leq\cdots\leq g_n(K)$, and that $g_n(K)=g_n(J)$ if $K$ and $J$ are concordant.  Also, $K$ is slice if and only if $g_n(K)=0$ for some $n\geq -1$.

Now we turn to the algebraic definitions.  If $G$ is a group, the \textbf{derived series of $G$} is defined recursively by $\Gn{0}=G$ and $\Gn{i+1}=\left[\Gn{i},\Gn{i}\right]$.  The \textbf{rational derived series of $G$} is defined recursively by setting $\Gr{G}{0}=G$ and $\displaystyle\Gr{G}{i+1}=\left\{g\in G: g^k\in\left[\Gr{G}{i},\Gr{G}{i}\right], \textrm{ for some } k>0\right\}$.
\ \\
\begin{defn}~\cite{COT03}
	Let $M$ be closed, orientable 3-manifold.  A spin 4-manifold $W$ with $\del W=M$ is an \textbf{$(n)$-solution for $M$} if the inclusion-induced map $i_\ast:\H{1}{M}\to\H{1}{W}$ is an isomorphism and if there are embedded surfaces $L_i$ and $D_i$ (with product neighborhoods) for $i=1,\ldots,m$ that satisfy the following conditions:
	\begin{enumerate}
		\item the homology classes $\left\{\left[L_1\right],\left[D_1\right],\ldots,\left[L_m\right],\left[D_m\right]\right\}$ form an ordered basis for $\H{2}{W}$,
		\item the intersection form $\big(\H{2}{W},\cdot\big)$ with respect to this ordered basis is a direct sum of hyperbolics,
		\item $L_i\cap D_j$ is empty if $i\neq j$,
		\item for each $i$, $L_i$ and $D_i$ intersect transversely at a point, and
		\item each $L_i$ and $D_i$ are \textbf{$(n)$-surfaces}, i.e. $\p(L_i)\subset\pn{W}{n}$ and $\p(D_i)\subset\pn{W}{n}$.
	\end{enumerate}
	If, in addition, $\p(L_{i})\subset\pn{W}{n+1}$ for each $i$, we say $W$ is an \textbf{$(n.5)$-solution} for $M$.  Since $H_1(W)$ has no 2-torsion and the intersection form of $W$ is even, $W$ is necessarily spin.
	
	If a closed, orientable 3-manifold has an $(n)$-solution, we say $M$ is \textbf{$(n)$-solvable}.  A knot $K$ in $S^3$ is an \textbf{$(n)$-solvable knot} if the zero surgery on $K$ is $(n)$-solvable.
\end{defn}

As in~\cite{COT03}, the set of all $(n)$-solvable knots is denoted $\F_n$, and Cochran-Orr-Teichner showed that the $\F_n$ form a nested series of subgroups of $\C$.  This series of subgroups is \textbf{the $(n)$-solvable filtration of the knot concordance group}.

Given a closed $3$-manifold and a homomorphism $\phi:\p(M)\to \Gamma$ where $\Gamma$ is any group, one can define the von Neumann $\rho$-invariant $\rho(M,\phi)\in\R$~\cite[Section 4]{CG85}.  See~\cite{CT07} for an analytical interpretation of these von Neumann $\rho$-invariants.

\begin{defn}\label{sigdef}
	For $K\in\F_n$, we define the \textbf{$n^{\textrm{th}}$-order signatures of $K$} to be the elements of the set ${\mathfrak S}^n(K)=\left\{\rho(M_K,\phi)\in\R\,| \ \phi:\pi(M_K)\xrightarrow{i_\ast}\pi\twoheadrightarrow
 \pi/\pi^{(n+1)}_r\right\}$ where $\pi=\p(W)$, $W$ is an $(n)$-solution for $M_K$, $i:M_K\to W$ is the inclusion map, and $\rho(M_{K},\phi)$ is the associated von Neumann $\rho$-invariant.  While this set of signatures is an isotopy invariant of $K$, it is not a concordance invariant~\cite[Example 3.2]{pHthesis}.
\end{defn}

Recall the Cheeger-Gromov estimate for the von Neumann $\rho$-invariants of a given closed, orientable 3-manifold~\cite{CG85}.  That is, given a closed, orientable 3-manifold $M$, there is a constant $C_M$ such that
\begin{equation}\label{cgbound}
	\left|\rho(M,\phi)\right|< C_M
\end{equation}
for all homomorphisms $\phi:\p(M)\to \Gamma$ for any group $\Gamma$.  Thus for a fixed knot $K$ and a fixed $n$, the set ${\mathfrak{S}}^{n}(K)$ is a bounded set of real numbers.

\section{Concrete examples in the smooth category}\label{smoothcategory}

In this section we work in the smooth category.  The purpose of this section is to construct non-slice knots that bound Gropes of a fixed height.  We compute the higher-order genera in these examples and conclude that for any positive integers $n$ and $m$, there is a knot whose smooth $n^{\textrm{th}}$-order genus is equal to $m$.  The computations do not make use of our $n^{\textrm{th}}$-order signatures.

Let $K$ denote any knot with non-negative maximal Thurston-Bennequin number.  For example, if $K$ is the right-handed trefoil, then $TB(K)=1$.  Let $D(K)$ denote the positively-clasped, untwisted Whitehead double of $K$ as depicted in Figure~\ref{whpos}.  For $i\geq 1$, let $D^{i}(K)=D(D^{i-1}(K))$ denote the $i^{\textrm{th}}$ iterated Whitehead double of $K$.  By Livingston~\cite{cL04}, we know that $TB(K)\geq0$ implies that the Ozsv\'{a}th-Szab\'{o} $\tau$-invariant of $D^i(K)$ is nontrivial, i.e. $\tau(D^{i}(K))=1$.  It follows from~\cite[Corollary 1.3]{OS03} that $D^{i}(K)$ is not smoothly slice for all $i\geq 1$.  It should be noted that earlier work of Lee Rudolph implies that $D^{i}(K)$ is not slice for all $i\geq 1$ if $K$ is the right-handed trefoil~\cite{lR93}.

\begin{figure}[ht!]
	\begin{center}
		\begin{picture}(142,142)(0,0)
			\put(60, 15){$K$}
			\includegraphics{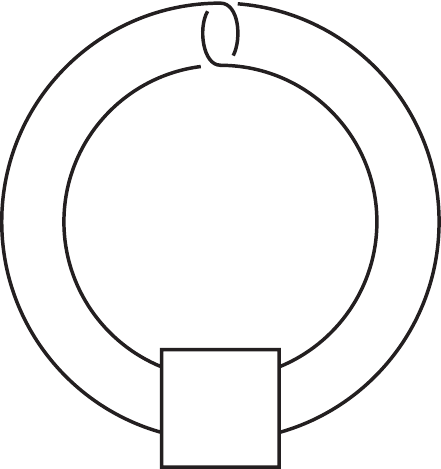}
		\end{picture}
		\caption{$D(K)$: the positively-clasped, untwisted Whitehead double of $K$.}
		\label{whpos}
	\end{center}
\end{figure}

We describe a Grope of height $2$ in $S^{3}\times I$ bounded by $D(K)$.  The standard Seifert surface for $D(K)$ has a symplectic basis of curves, each of which inherits the zero framing from this surface.  This basis is pictured in Figure~\ref{untwistedbasis}.  Let $\a$ denote the basis curve that ``goes over the bridge'' of this Seifert surface, and let $\b$ denote the other curve.  Pull $\a$ slightly out of the page so that the intersection point with $\b$ is removed.  Observe that the link $\a^{+}\cup\b$ is two parallel copies of $K$.  Now push these two curves down in the $I$ direction and glue parallel Seifert surfaces for $K$.  The Seifert surface for $D(K)$ together with the pushing annuli and Seifert surfaces for $K$ comprise a height $2$ Grope for $D(K)$ in $S^{3}\times I$.  The genus of the first stage of this Grope is $1$.  By~\cite[Corollary 1.3]{OS03} $1=\tau(D(K))\leq g_{-1}(D(K))\leq g_{0}(D(K))$ and $g_{0}(D(K))\leq 1$ by construction, we have $g_{0}(D(K))=1$.

\begin{figure}[ht!]
	\begin{center}
		\begin{picture}(142,142)(0,0)
			\put(60, 15){$K$}
			\includegraphics{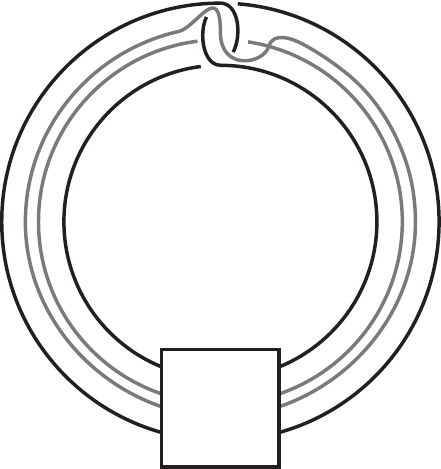}
		\end{picture}
		\caption{A basis of untwisted curves for the Seifert surface of $D(K)$.}
		\label{untwistedbasis}
	\end{center}
\end{figure}

We can iterate this procedure to build a Grope of height $n+1$ in $S^{3}\times I$ bounded by $D^{n}(K)$, and the first stage of this Grope has genus $1$.  As before, we have $1\leq\tau(D^{n}(K))\leq g_{-1}(D^{n}(K)\leq g_{0}(D^{n}(K))\leq\cdots\leq g_{n-1}(D^{n}(K))\leq 1$, whence $g_{n-1}(D^{n}(K))=1$.

Since $\tau:\C\to\Z$ is a homomorphism~\cite[Theorem 1.2]{OS03}, we conclude that $g_{n-1}\left(\#_{m} D^{n}(K)\right)\geq\tau\left(\#_{m} D^{n}(K)\right)=m\cdot\tau\left(D^{n}(K)\right)=m$ and $g_{n-1}\left(\#_{m} D^{n}(K)\right)\leq m$ by construction.  To summarize, we have the following theorem.

\begin{thm}
	For any $n\geq 0$ and $m\geq 1$, there is a knot $K\in\G^{smooth}_{n+2}$ of infinite order, and $g_{n}(K)=m$.
\end{thm}

\begin{rem}
	Since the Alexander polynomial of $D(K)$ is trivial, it can be shown that $D(K)$ is smoothly $(n)$-solvable for all $n$.  However, whether $D(K)\in\G^{smooth}_{n+2}$ for all $n$ is still an open question.
\end{rem}

\section{Lower bounds on higher-order genera}

We now turn to our higher-order signatures as tools for estimating the higher-order genera.  While the higher-order signatures are not explicitly computable, we demonstrate how to ensure that all higher-order signatures are large enough to guarantee that the higher-order genera are large.

\begin{lem}\label{sigsboundbetti}
	Let $K\in\F_{n}$ and $W$ be an $(n)$-solution for $M_{K}$.  Then the $n^{\textrm{th}}$-order signature of $K$ associated to $W$ satisfies $\left|\rho\left(M_{K},\phi\right)\right|\leq \b_{2}(W)$.
\end{lem}
\begin{proof}
	Let $\phi:\p(M_K)\xrightarrow{i_\ast} \p(W)\xrightarrow{\pi} \p(W)/\pr{W}{n+1}$.  By the definition of an $(n)$-solution, the ordinary intersection form of $W$ is a direct sum of hyperbolics, implying that the ordinary signature of $W$ is zero.  Since $\phi$ factors through $\p(W)$, we have that $$\rho(M_K,\phi)=\sigma^{(2)}\left(W,\p(W)/\pr{W}{n+1}\right)-\sigma(W)=\sigma^{(2)}\left(W,\p(W)/\pr{W}{n+1}\right)$$  Here $\sigma^{(2)}(W,--)$ refers to the $L^{2}$-signature of $W$ associated to the quotient $\p(W)/\pr{W}{n+1}$.  We refer the reader to Section 5 of~\cite{COT03} for a thorough explanation of $L^{2}$-signatures.  Cha has shown that $\left|\sigma^{(2)}\left(W,\p(W)/\pr{W}{n+1}\right)\right|\leq\beta_2(W)$~\cite[Lemma 2.7]{jC08}.
\end{proof}  

That the homomorphism $\phi:\p(M_{K})\to\p(W)/\pr{W}{n+1}$ factors through $\p(W)$ of bounding $4$-manifold $W$ is crucial.  Our philosophy differs from Cha's~\cite{jC08} in that we assume our homomorphisms factor through bounding $4$-manifolds (cf. Definition~\ref{sigdef}), whereas Cha takes a homomorphism $\p(M_{K})\to\Gamma$ and tries to extend it to a bounding $4$-manifold.  In particular, Cha finds a homomorphism $\phi_{\sigma}:\p(M_{K})\to\Z$ that factors through a certain bounding $4$-manifold, and the von Neumann $\rho$-invariant associated to this homomorphism satisfies $|\rho(M_{K},\phi_{\sigma})|\leq4\,g_{-1}(K)$, where $g_{-1}(K)$ is the slice genus of $K$~\cite[Theorem 1.1 and Proposition 1.2]{jC08}.  We, however, consider many homomorphisms that we assume extend to bounding $4$-manifolds, and we show that (at least) one of the associated $\rho$-invariants satisfies $|\rho|\leq 4\,g_{n}(K)$, where $g_{n}(K)$ is the $n^{\textrm{th}}$-order genus of $K$.

\begin{thm}\label{sigboundgen}
	If $K\in\G_{n+2}$, one of the $n^{\textrm{th}}$-order signatures $\rho\in\mathfrak{S}^n(K)$ satisfies $|\rho|\leq4\,g_n(K)$.
\end{thm}
\begin{proof}
	Let $\Sigma$ be the first stage of a Grope of height $n+2$ that realizes $g_n(K)$, i.e. $g(\Sigma)=g_{n}(K)$.  Cochran-Orr-Teichner construct an $(n)$-solution $W$ by surgering $\Sigma$, and $\beta_2(W)=4\,g(\Sigma)=4\,g_n(K)$~\cite[Theorem 8.11]{COT03}.  The conclusion follows from Lemma~\ref{sigsboundbetti}.
\end{proof}

\begin{rem}
	Theorem~\ref{sigboundgen} may be thought of as a higher-order analogue of Murasugi's inequality~\cite[Theorem 9.1]{kM65}.  Unlike the subsequent inequalities of Gilmer~\cite[Theorem 1]{pG82} and Cha~\cite[Proposition 5.1]{jC08}, our result gives higher-order obstructions to the higher-order genera.
\end{rem}

\begin{cor}
	If $K$ is a slice knot, then for any $n$, one of the $n^{\textrm{th}}$-order signatures of $K$ vanishes.
\end{cor}

\begin{prop}\label{n.5obstruction}
	Suppose $K$ is $(n)$-solvable.  If $K$ is $(n.5)$-solvable, then one of the $n^{\textrm{th}}$-order signatures of $K$ vanishes.
\end{prop}
\begin{proof}
	Let $W$ be an $(n.5)$-solution for $K$. It follows from~\cite[Theorem 4.2]{COT03} that the $n^{\textrm{th}}$-order signature of $K$ associated to $W$ vanishes.
\end{proof}

\begin{rem}
	The conclusion holds even if $K$ is assumed to be merely rationally $(n.5)$-solvable~\cite[Definition 4.1]{COT03}.
\end{rem}

If the Alexander polynomial of a knot is trivial, then the knot is topologically slice~\cite{FQ90}.  In particular, Alexander polynomial one knots are $(n)$-solvable for all $n$.  Consequently, the $n^{\textrm{th}}$-order signatures of an Alexander polynomial one knot are all equal to the classical signature, namely zero.  As the $n^{\textrm{th}}$-order signatures are topological invariants, they will not give accurate bounds for the smooth higher-order genera.  For example, the knots constructed in Section~\ref{smoothcategory} had trivial Alexander polynomial but large smooth $n^{\textrm{th}}$-order genera.

\begin{thm}\label{arbhinog}
	For any $n\geq 1$, there is a fixed $g$ and a knot in $\G_{n+2}$ with slice genus bounded above $g$ and arbitrarily high $n^{\textrm{th}}$-order genus.  Furthermore, this knot has infinite order in $\G_{n+2}/\F_{n.5}$.
\end{thm}

\begin{rem}
	The statement of Theorem~\ref{arbhinog} seems to be false for $n= 0$.  For example, if $K\in\G_{2}$, one can construct a Grope of height $2$ bounded by $K$ whose first stage has genus equal to the Seifert genus of $K$.  See~\cite[Remark 8.14]{COT03} for a discussion.
\end{rem}

\begin{proof}
	We construct knots according to Cochran-Orr-Teichner~\cite{COT03} and Cochran-Teichner~\cite{CT07}.  We borrow the knot $J$ from~\cite[Figure 3.6]{CT07}.  Let $J_m=\#_{m} J$; then $J_m$ bounds a Grope of height $2$ (and is $(0)$-solvable), and $\rho_0(J_m)=\frac{4m}{3}$~\cite[Lemma 4.5]{CT07}.  Let $R$ denote the knot pictured in Figure~\ref{ribbon} (ignore the curve $\eta$ for now).  $R$ is a fibered, genus $2$, ribbon knot~\cite[p. 447]{COT03}.

\begin{figure}[h!]
	\begin{center}
		\begin{picture}(300, 150)(0, 0)
			\includegraphics[bb=0 0 260 149]{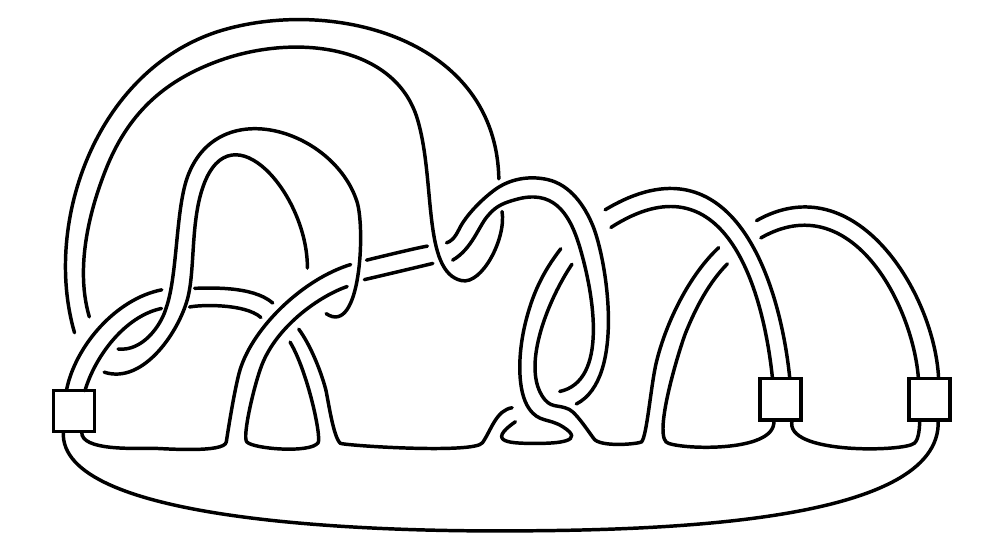}
			\put(-243, 36){\small -1}
			\put(-39, 39){\small -2}
			\put(2, 39){\small +1}
			\put(-9, 7){\small $R$}
			\put(-190, 127){\small $\eta$}
		\end{picture}
		\caption{The ribbon knot $R$ and a curve $\eta\in\pn{S^3-R}{2}$.}
		\label{ribbon}
	\end{center}
\end{figure}
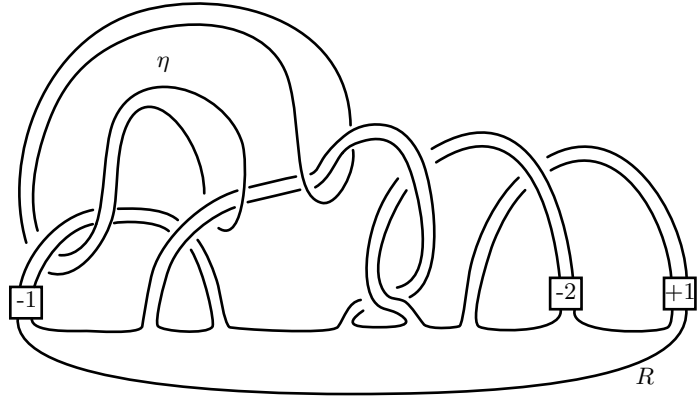

By~\cite[Theorem 4.3]{CT07}, there is a collection of unknotted curves $\eta_i$, $1\leq i\leq j$, in $S^3-R$ with $[\eta_i]\in\pn{M_R}{n}$ and for any $(n)$-solution $V$ of $M_R$, some $i_\ast\left(\left[\eta_k\right]\right)\not\in\pr{V}{n+1}$.  For example, Figure~\ref{ribbon} shows an unknotted curve $\eta$ whose homotopy class lies in $\pn{S^3-R}{2}\cong\pn{M_R}{2}$, and this curve never maps into $\pr{V}{3}$ for any $(2)$-solution $V$ for $M_R$~\cite[Theorem 4.2]{COT03}.  Let $K=K_m$ denote the knot obtained by infecting $R$ by $J_m$ along $\eta_i$ (for each $i$).

Infecting $R$ by $J_{m}$ along $\eta_{i}$ means to grab the strands of $R$ passing through the unknotted curve $\eta_{i}$ and tie them collectively into the knot $J_{m}$.  Below is a schematic diagram of the infection operation.

\begin{figure}[ht!]
	\begin{center}
		\begin{picture}(140, 90) (0,0)
			\put(-8, 40){$\eta_{i}$}
			\put(47, 60){$R$}
			\put(62, 37){\vector(1,0){20}}
			\includegraphics{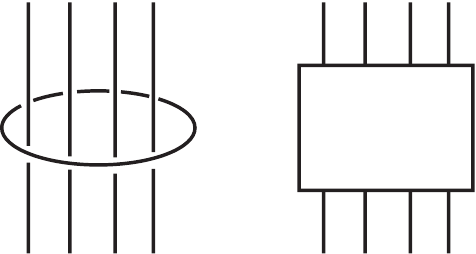}
			\put(-30, 34){$J_{m}$}
		\end{picture}
		\caption{Infecting $R$ by $J_{m}$ along $\eta_{i}$}
		\label{infection}
	\end{center}
\end{figure}
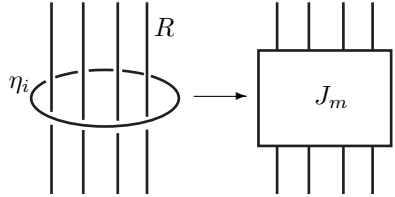

We claim that by choosing $m$ sufficiently large, we can guarantee all $\rho\in\mathfrak{S}^n(K)$ are arbitrarily large. It will follow from Theorem~\ref{sigboundgen} that $K$ will have arbitrarily large $n^{\textrm{th}}$-order genus, modulo verifying $K$ bounds a Grope of height $n+2$, which Cochran and Teichner proved in~\cite[Theorem 3.8]{CT07}.

Since the $\eta_i$ have linking number zero with $R$, we can take a Seifert surface for $R$ and tube around the $\eta_i$ so that the tubes are disjoint. We are left with a Seifert surface for $R$ which the $\eta_i$ do not intersect.  The knot $K$ will have genus bounded above by the genus of our tubed surface for $R$.  We now explain how to increase the $n^{\textrm{th}}$-order genus of $K$ without increasing the genus.

Since our $J_m$ are $(0)$-solvable, let $W_m$ denote a $(0)$-solution for $J_m$.  We form a 4-manifold $E$ from $$ M_R\times[0,1]\bigsqcup_{i=1}^j -M_{J_m}\times[0,1]$$ by identifying, for each $i$, the copy of $\eta_i\times D^2$ in $M_R\times\{1\}$ with the tubular neighborhood of $J_m$ in $M_{J_m}\times\{0\}$ as in Figure~\ref{buildsolution}.  The dashed arcs represent the solid tori $\eta_i\times D^2$.  As indicated in Figure~\ref{buildsolution}, $\del E=M_R\sqcup -M_K\sqcup M_{J_m}\sqcup\cdots\sqcup M_{J_m}$.  We form another 4-manifold $C$ from $E$ by gluing a copy of $W_m$ to each $M_{J_m}\subset\del E$.

Now let $W$ be any $(n)$-solution for $M_K$.  Let $V=C\cup_{-M_K} -W$ so that $\del V=M_R$.  Then $V$ is an $(n)$-solution for $M_R$~\cite[Proof of Theorem 4.2]{CT07}.  From our previous discussion, there is a $\eta_k$ with $i_\ast\left(\left[\eta_k\right]\right)\not\in\pr{V}{n+1}$.  Since $\eta_k$ lives in $M_K$, we may include $\eta_k$ into $W$.  Since $W\subset V$, $i_\ast\left(\left[\eta_k\right]\right)\not\in\pr{W}{n+1}$.

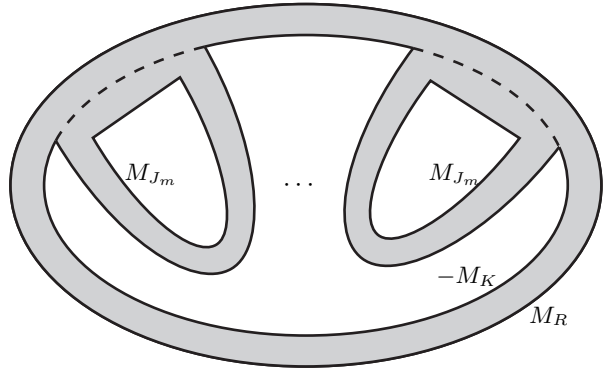
\begin{figure}[h!]
	\begin{center}
		\begin{picture}(220, 140)(0,0)
			\includegraphics[bb=0 0 224 138, scale=1]{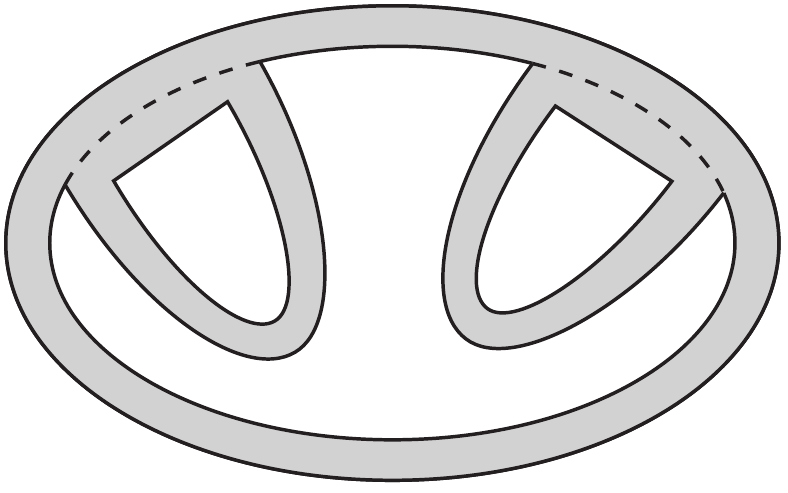}
			\put(-120, 70){$\cdots$}
			\put(-27, 20){\small $M_R$}
			\put(-62, 35){\small $-M_K$}
			\put(-180, 74){\small $M_{J_m}$}
			\put(-65, 74){\small $M_{J_m}$}
		\end{picture}
		\caption{The 4-manifold $E$.}
		\label{buildsolution}
	\end{center}
\end{figure}

Consider the homomorphism $\phi:\p(M_K)\xrightarrow{i_\ast}\p(W)\xrightarrow{\pi}\p(W)/\pr{W}{n+1}$.  Let $\Gamma=\p(W)/\pr{W}{n+1}$.  Now $M_R-\left(\sqcup \eta_i\right)\subset M_K$, so $\phi$ induces a homomorphism $\phi':\p(M_R-\left(\sqcup\eta_i\right))\to\Gamma$.  Since $M_R$ is obtained by $M_R-\left(\sqcup\eta_i\right)$ by adding $j$ 2-cells along the meridians of the $\eta_i$ and then by adding $j$ 3-cells, this $\phi'$ will extend to a homomorphism $\phi_R:\p(M_R)\to\Gamma$ if the meridians of the $\eta_i$ die under $\phi$.  Now $\eta_i\in\pn{M_R}{n}$ and $\Gamma^{(n+1)}=1$, so~\cite[Theorem 8.1]{tC04} implies that $\eta_i\in\pn{M_K}{n}$.  Since the meridian $\mu_i$ of each $J_m$ is identified with the longitude of $\eta_i$, $\mu_i\in\pn{M_K}{n}$.  Thus $\phi(\mu_i)\in\Gamma^{(n)}$.  Since $\mu_i$ generates $\p(S^3-J_m)/\pn{S^3-J_m}{1}$, we see $\phi\left(\pn{S^3-J_m}{1}\right)\subset\Gamma^{(n+1)}=1$.  In particular the meridian of each $\eta_i$ dies under $\phi$, and hence $\phi'$ extends to a map $\phi_R:\p(M_R)\to\Gamma$.

By~\cite[Proposition 4.4]{CT07}, the $\rho$-invariants of $M_K$ and $M_R$ are related by $$\rho(M_K,\phi)-\rho(M_R,\phi_R)=\sum_{i=1}^j \epsilon_i \rho_0(J_m)$$ where $\epsilon_i=0$ or $1$ according to whether $\phi_R\left(\left[\eta_i\right]\right)=1$ or not.  We argued that previously that $i_\ast\left(\left[\eta_k\right]\right)\not\in\pr{W}{n+1}$, so $\phi_R\left(\left[\eta_i\right]\right)\neq1$.  Recall that the set of $\rho$-invariants of $M_R$ are bounded above by the Cheeger-Gromov constant $C_{M_R}$ (cf. equation~\ref{cgbound}).  Thus, by choosing $m$ sufficiently large, we will obtain a knot $K$ with $\left |\rho(M_K,\phi)\right|>B$ for some large constant $B$.  Since $W$ was an arbitrary $(n)$-solution for $K$, we have proved that every $n^{\textrm{th}}$-order signature for $K$ is larger than $B$.  Appealing to Theorem~\ref{sigboundgen}, we see that $g_n(K)$ is arbitrarily large.  We should note here that since $0\not\in\mathfrak S^{n}(K)$, $K\not\in\mathcal{F}_{n.5}$ (by Proposition~\ref{n.5obstruction}).  \cite[Theorem 4.2]{CT07} establishes that $K$ has infinite order in $\G_{n+2}/\F_{n.5}$.
\end{proof}

\begin{cor}\label{nogisbetterthanslicegenus}
	Given any $n\geq 1$, there exist infinitely many knots in $\G_{n+2}$ whose slice genus agree but whose $n^{\textrm{th}}$-order genera are distinct.
\end{cor}
\begin{proof}
	By Theorem~\ref{arbhinog}, there is a positive integer $g$ and a sequence $\{K_{i}\}_{i=1}^{\infty}$ of knots in $\G_{n+2}$ with $g_{-1}(K_{i})\leq g$ and $g_{n}(K_{i})<g_{n}(K_{i+1})$ for all $i\geq 1$.  Since the set $\{g_{-1}(K_{i})\}$ is a finite set, we can pass to a subsequence of knots with the same slice genera but different $n^{\textrm{th}}$-order genera.
\end{proof}

\begin{rem}
	We can improve the statement of Corollary~\ref{nogisbetterthanslicegenus} to say that for each $n\geq 2$, there are infinitely many knots in $\G_{n+2}$ with identical $i^{\textrm{th}}$-order genera for $i\leq n-1$ and distinct $n^{\textrm{th}}$-order genera.  However, the proof is too lengthy to include in this paper.  We refer the reader to the author's thesis for a proof~\cite[Theorem 5.4]{pHthesis}.  This result implies that the lower-order genera of knots are inadequate measures of the complexity of $\G_{n+2}$ and that the higher-order genera capture some of the missed information.  Examples of this phenomenon can be constructed by infection on the $9_{46}$ knot as in~\cite{pH08}.
\end{rem}

\begin{exam}
	We provide a concrete family of examples of knots $\{L_m\}_{m=1}^\infty$ in $\G_{3}$ with slice genus bounded above by $3$ and for any $C\in\N$ there is a positive integer $N$ such that for all $n\geq N$, $g_{1}(L_n)> C$.  Our family is inspired by Cochran-Harvey-Leidy's family $J_n$ (cf.~\cite{CHL07b}).



Cochran-Harvey-Leidy defined their knots by infecting along the curves $\alpha$ and $\beta$ in Figure~\ref{infectingcurves}.  We cannot use these curves for the purpose of constructing knots bounding Gropes because the two punctured tori bounded by $\alpha$ and $\beta$ intersect.  As per~\cite[Lemma 3.9]{CT07}, we find curves $\alpha'$ and $\beta'$ that are homotopic to $\alpha$ and $\beta$, respectively, and that bound disjoint height $1$ Gropes in $S^3-R$.  Since these curves are homotopic, the $n^{\textrm{th}}$-order signatures will not distinguish our examples from the examples of~\cite{CHL07b}.  However, our examples are probably not concordant to theirs.

\begin{figure}[h!]
	\begin{center}
		\begin{picture}(330, 160)(0,0)
			\put(26, 65){\small$\alpha$}
			\put(104, 65){\small$\beta$}
			\put(188, 48){\small$\alpha'$}
			\put(195, 90){\small$\beta'$}
			\put(200, 100){\vector(1,3){4}}
			\includegraphics[bb=0 0 777 398, scale=.4]{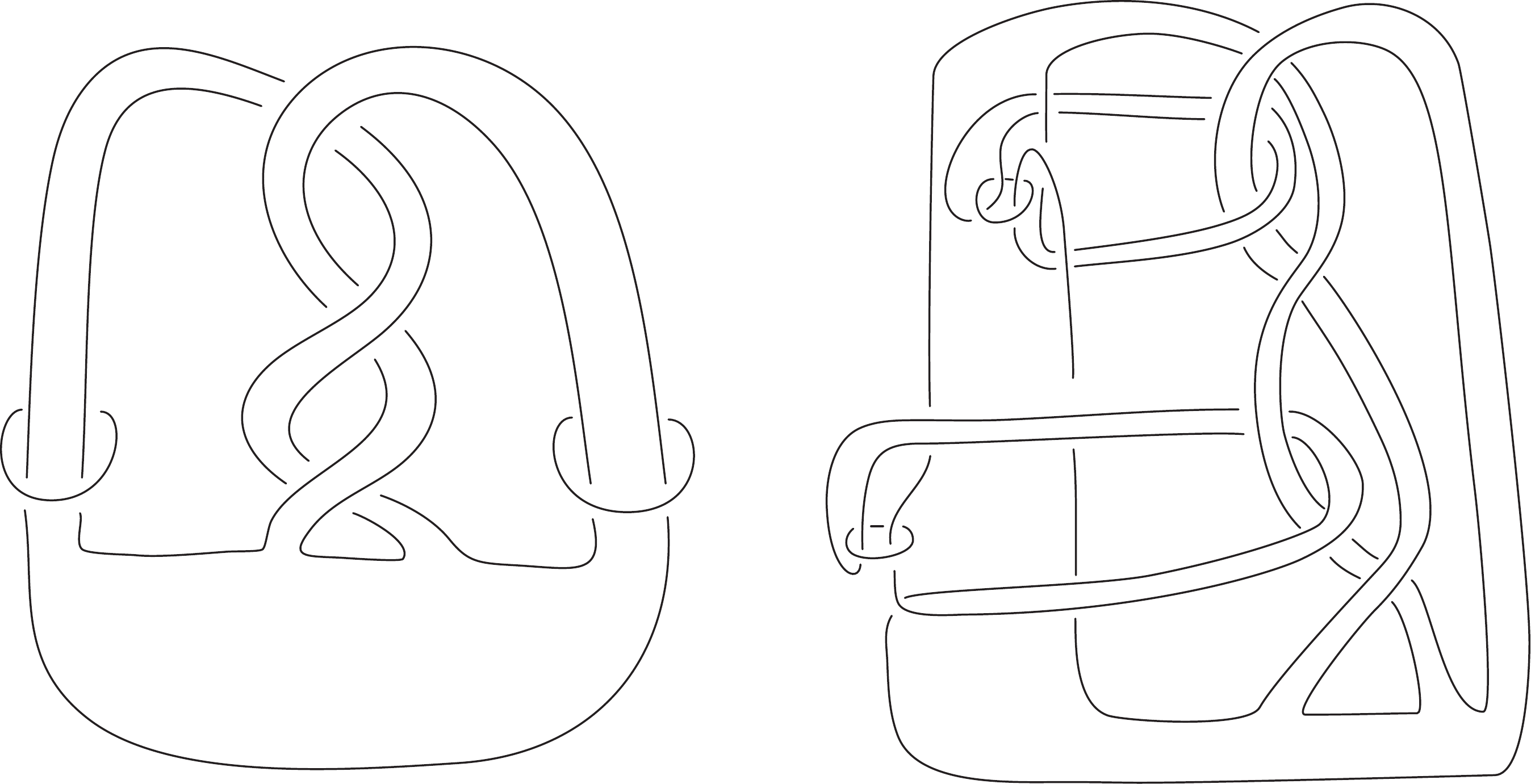}
		\end{picture}
		\caption{The infection curves $\alpha$ and $\beta$, and homotopic infection curves $\alpha'$ and $\beta'$.}
		\label{infectingcurves}
	\end{center}
\end{figure}
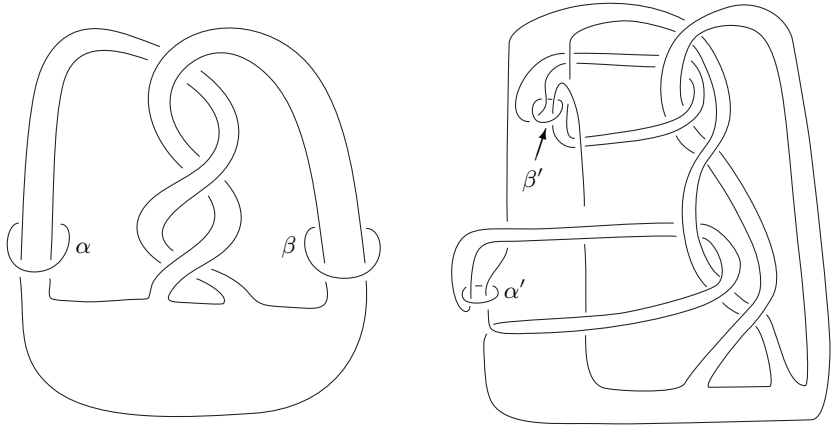

Now, let $J$ be the knot from~\cite{CT07} and let $J_m=\#_m J$.  $J_m$ no longer refers to the knots from~\cite{CHL07b}.  Let $L_m$ be infection on $R=9_{46}$ along $\alpha'$ and $\beta'$ by $J_m$.  We chose $\alpha'$ and $\beta'$ so that they bound disjointly embedded punctured tori in the complement of $R$, so by~\cite[Proof of Theorem 3.7]{CT07} the knots $L_m$ will bound Gropes of height $3$ in $D^4$.  Since $\alpha'$ and $\beta'$ lie off of a genus 3 Seifert surface for $R$, $L_m$ will have slice genus less than or equal to three.

Let $V$ be a $(1)$-solution for $M=M_{L_m}$.  Let $\pi=\p(V)$. Since $H_1(V)\cong \Z$ is torsion-free, we conclude $H_1(V)\cong\pi/\pi^{1}\cong\pi/\pi_{r}^{1}\cong\Z$.  Let $\phi:\p(M)\xrightarrow{i_\ast}\pi\twoheadrightarrow\pi/\pi_{r}^{1}$.  Since $i_\ast:H_1(M)\xrightarrow{\cong}H_1(V)\cong\pi/\pi_{r}^{1}$, we see that $\phi:\p(M)\twoheadrightarrow H_1(M)\xrightarrow{i_\ast\cong} H_1(V)$.  For emphasis, let $H_1(M;\Q[s,s\inv])$ denote the first homology of the infinite cyclic cover of $M$ as a $\Q[s,s\inv]$-module, where $H_1(M)=\langle s\rangle$, and let $H_1(M;\Q[t,t\inv])$ denote the first homology induced by the coefficient system $\phi:\p(M)\to\pi/\pi_{r}^{1}$.  The curves $\alpha$ and $\beta$ generate $H_1(M;\Q[s,s\inv])$, and since $\alpha'$ and $\beta'$ are homotopic to these generators, $\alpha'$ and $\beta'$ also generate $H_1(M;\Q[s,s\inv])$.  Since the coefficient system $\phi$ is $\p(M)\twoheadrightarrow H_1(M)$ followed by an isomorphism, $\alpha'$ and $\beta'$ generate $H_1(M;\Q[t,t\inv])$.

Cochran-Orr-Teichner proved that the coefficient system $\phi$ induces a hyperbolic bilinear form $Bl(\cdot,\cdot)$ defined on $H_1(M;\Q[t,t\inv])$~\cite[Theorem 2.13]{COT03} and that $${\mathfrak{k}}:=\ker\{i_\ast: H_1(M;\Q[t,t\inv])\to H_1(V;\Q[t,t\inv])\}$$ satisfies ${\mathfrak{k}}={\mathfrak{k}}^\perp$ with respect to this form~\cite[Theorem 4.4]{COT03}.  Since this form is hyperbolic and $\a'$ and $\b'$ generate $H_{1}(M:\Q[t,t\inv])$, $Bl(\alpha',\beta')$ is nonzero, and hence one of $\alpha'$ and $\beta'$ is not in $\mathfrak{k}$.  By the bilinearity of $Bl$, all integer multiples of $\alpha'$ or $\beta'$ are not in $\mathfrak{k}$.  Recall that $H_{1}(V;\Q[t,t\inv])$ is the first homology of the infinite-cyclic cover $\widetilde V$ of $V$, viewed as a $\Q[t,t\inv]$-module, and $\p\left(\widetilde V\right)=\pn{V}{1}$.  If $\alpha'$ were to map to zero in $H_{1}(V;\Q[t,t\inv])$, then $\alpha'$ would map into $\pn{V}{2}$.  Since no multilple of $\alpha'$ (or of $\beta'$) lie in $\mathfrak k$, we conclude that $\alpha'$ or $\beta'$ does not map into $\pr{V}{2}$.  As in Theorem~\ref{arbhinog}, we have the following relationship between the $\rho$-invariants:
$$\rho(M,\phi)-\rho(M_R,\phi_R)=\epsilon_{\alpha'}\rho_0(J_m)+\epsilon_{\beta'}\rho_0(J_m)$$

Since one of $\alpha'$ and $\beta'$ does not map into $\pr{V}{2}$, one of $\epsilon_{\alpha'}$ or $\epsilon_{\beta'}$ is one, as discussed in the proof of Theorem~\ref{arbhinog}.  By choosing $m$ sufficiently large, the number $|\rho(M,\phi)|$ can be made arbitrarily large.  Since $V$ was an arbitrary $(1)$-solution, we have that $g_1(L_m)$ is arbitrarily large by Theorem~\ref{sigboundgen}.
\end{exam}

\section{Applications to a Geometric Structure on the Grope Filtration}

Let $B^{r}_{n}$ denote the subset of all $K$ in $\G_{n+2}$ such that $g_{n}(K)\leq r$.  Since $g_{-1}\leq g_{0}\leq \cdots\leq g_{n}$, we see that $B_{-1}^{r}\supseteq B_{0}^{r}\supseteq \cdots \supseteq B_{n}^{r}$.  Our main result (Theorem~\ref{arbhinog}) is that the higher-order genera are finer measures than the slice genus.  Furthermore, by the remark after Corollary~\ref{nogisbetterthanslicegenus}, the $n^{\textrm{th}}$-order genus is a finer measure than the lower-order genera, up to order at least $n-2$.    That is, some (depending on $n$ and $r$) of these subset containments are proper.  Consequently, these higher-order genera provide a further refinement of the Grope filtration of the knot concordance group.  That is, after determining how deep a knot lies in the Grope filtration (say in $\G_{n+2}$), one might try to determine the knot's $n^{\textrm{th}}$-order genus.

We attempt to complement these comments with the diagram in Figure~\ref{refinement}.  The ambient three-dimensional space represents $\G_{n+2}$, the plane represents $\G_{n+3}$, the line represents $\G_{n+4}$, and the origin represents $\displaystyle\bigcap_{n\geq 0} \G_{n}$.  The corresponding balls have been drawn.  The diagram suggests the existence of knots in $B^{r}_{n}-B^{r}_{n+1}$, which was proven in Theorem~\ref{arbhinog} and Corollary~\ref{nogisbetterthanslicegenus} for certain $n$ and $r$.
\newpage
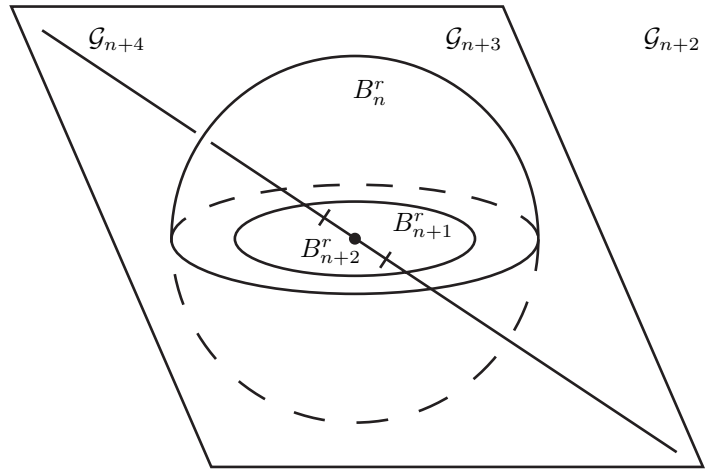
\begin{figure}[ht!]
	\begin{center}
		\begin{picture}(263, 185)(0,0)
			\put(240, 160){$\G_{n+2}$}
			\put(165, 160){$\G_{n+3}$}
			\put(30, 160){$\G_{n+4}$}
			\put(130, 140){$B_{n}^{r}$}
			\put(145, 90){$B_{n+1}^{r}$}
			\put(110, 80){$B_{n+2}^{r}$}
			\includegraphics{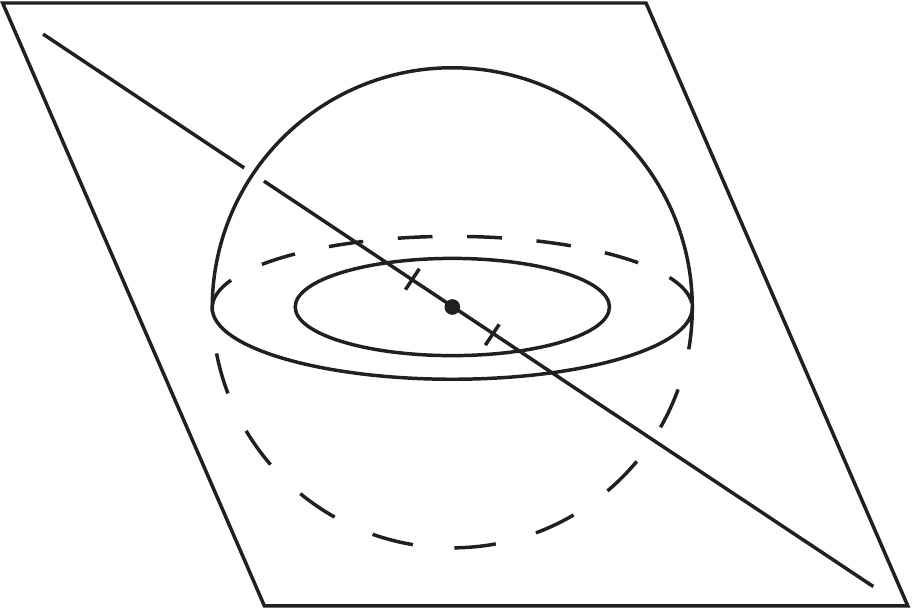}
		\end{picture}
		\caption{The refinement of the Grope filtration by the higher-order genera.}
		\label{refinement}
	\end{center}
\end{figure}

%
%
%
%

	\bibliographystyle{amsalpha}
	\bibliography{hogbib}

\end{document}